\documentclass[twoside]{amsart}
\usepackage{amssymb}
\usepackage{amsmath,amsthm}
\usepackage{amsfonts}
\usepackage{graphicx}
\usepackage{amsmath}
\usepackage[all]{xy}
\usepackage{array}
\usepackage{xcolor}
\usepackage{mathtools}
\usepackage{pgf}
\usepackage{tikz-cd}
\usetikzlibrary{arrows,automata}
\usepackage[latin1]{inputenc}
\providecommand{\U}[1]{\protect\rule{.1in}{.1in}}
\newtheorem{theorem}{Theorem}

\newtheorem{corollary}[theorem]{Corollary}

\newtheorem{definition}[theorem]{Definition}

\newtheorem{lemma}[theorem]{Lemma}
\newtheorem{notation}[theorem]{Notation}

\newtheorem{remark}[theorem]{Remark}

\def\Rhat{\widehat{R}}

\def\id{{\rm id\,}}     

\def\Ext{{\rm Ext}\,}

\def\Ker{{\rm Ker}\,}

\def\l{\lambda}
\def\b{\beta}
\def\g{\gamma}
\def\kk{\kappa}

\def\restr{\lceil}

\def\a1{\aleph_1}                   

\def\restr{\upharpoonright}

\def\id{\mathop{\rm id}\nolimits}
\def\Ext{\mathop{\rm Ext}\nolimits}

\def\Ker{\mathop{\rm Ker}\nolimits}

\def\max{\mathop{\rm max}\nolimits}

\def\Z{{\mathbb Z}}
\def\Q{{\mathbb Q}}

\def\bS{{\mathbb S}}

\def\bar{\overline }

\def\etabar{{\bar \eta}}

\def\nubar{{\bar \nu}}

\def\a{\alpha}
\def\b{\beta}
\def\g{\gamma}

\def\om{\omega}
\def\va{\varphi}

\def\n{\nu}

\def\l{\lambda}

\def\d{\delta}

\def\aln{{\aleph_0}}
\def\al-n{{\aleph_n}}

\def\n+1d{{}^{ n+1 \downarrow }\l}

\def\size#1{\left|\,#1\,\right|}

\def\Zhat{\widehat \Z}

\def\Rhat{\widehat R}

\def\Bhat{\widehat B}

\def\restl{\mathop{\upharpoonleft}}
\def\restr{\mathop{\upharpoonright}}
\def\to{\rightarrow}

\def\lup{{}^{\omega \uparrow }\l}
\def\lupf{{}^{\omega \uparrow > }\l}

\begin{document}

\title{$\aleph_k$-free cogenerators}

\author{Manfred Dugas}
\address[Manfred Dugas]{Department of Mathematics, Baylor University, One Bear Place
\#97328, Waco, TX 76798-7328, USA}
\email{manfred\_dugas@baylor.edu}

\author{Daniel Herden}
\address[Daniel Herden]{Department of Mathematics, Baylor University, One Bear Place
\#97328, Waco, TX 76798-7328, USA}
\email{daniel\_herden@baylor.edu}

\author{Saharon Shelah}
\address[Saharon Shelah]{The Hebrew University of Jerusalem, Einstein Institute of
Mathematics, Edmond J. Safra Campus, Givat Ram, Jerusalem 91904, Israel\linebreak
\mbox{} \quad Department of Mathematics, Hill Center-Busch Campus, Rutgers, The State
University of New Jersey, 110 Frelinghuysen Road, Piscataway, NJ 08854-8019, USA}
\email{shelah@math.huji.ac.il}

\thanks{The third author was partially supported by ERC grant 338821.
The second and third author also thank NSF grant 1833363 for support.
This is DgHeSh:1171 in the third author's list of publications.}

\subjclass[2010]{Primary 13C05, 13C10, 20K20, 20K25, 20K35;
Secondary 03E05, 03E35}
\keywords{locally free groups, cotorsion groups, prediction principles}
\dedicatory{ }

\begin{abstract}
We prove in ZFC that an abelian group $C$ is cotorsion if and only if $\Ext(F,C) = 0$ for every $\aleph_k$-free group $F$,
and discuss some consequences and related results. This short note includes a condensed overview of the $\bar\lambda$-Black Box
for $\aleph_k$-free constructions in ZFC.
\end{abstract}

\maketitle

\tableofcontents

\section{Introduction}
In the theory of abelian groups, locally free groups and their properties have been the subject of extensive research.
In particular, for any given uncountable cardinal $\kappa$, we will call a group $G$ \emph{$\kk$-free} if every subgroup
$H\subseteq G$ of cardinality $|H|<\kk$ is free. One of the earliest and easiest examples \cite{B,Sp} of a non-free $\aleph_1$-free
group is the Baer-Specker group $\mathbb{Z}^\omega$, the cartesian product of countably infinitely many copies of the integers
$\mathbb{Z}$, and the cartesian product $\mathbb{Z}^\lambda$ is $\aleph_1$-free for any cardinal $\lambda$.
Apart from that, explicit examples of non-free $\kappa$-free groups are fairly difficult to come by and
require either some elaborate use of infinite combinatorics or of specific models of set theory. For instance, it is known that
every Whitehead group is $\aleph_1$-free \cite{St}, but the question whether non-free Whitehead groups exist is undecidable
and depends on the chosen model of set theory \cite{Ek,Sh44}. In G\"odel's Universe V=L, non-free $\kappa$-free groups
exist for all uncountable cardinals $\kappa$, and $\kappa$-free groups with prescribed properties are traditionally constructed
with help of the Jensen diamond principle $\diamondsuit$. Similarly, assuming only ZFC, the construction of $\aleph_1$-free
groups with various additional properties is possible utilizing Shelah's Black Box. See \cite{EM,GT} for some standard literature
on these constructions.

In contrast to this, hardly anything has been known about the existence of $\kappa$-free groups in ZFC for $\kappa>\aleph_1$.
Some first sporadic examples of non-free $\aleph_k$-free groups for integers $k\ge 2$ can be found in \cite{Gr,Hi},
however, the breakthrough in constructing $\aleph_k$-free groups with prescribed additional properties is more recent. In
\cite{GS3,S3}, $\aleph_k$-free groups with trivial dual were constructed, and \cite{GHS2} provides a construction for
$\aleph_k$-free groups with prescribed endomorphism rings. Similar constructions of $\aleph_k$-free groups and modules for
$k\ge 2$ can be found in \cite{GHSa,GSS,Habil,HSa} and are based on the $\bar\lambda$-Black Box as a guiding combinatorial principle.
For cardinals $\kappa \ge \aleph_\omega$, the situation concerning $\kappa$-free groups becomes considerably more complicated.
In~\cite{S4}, a construction for $\aleph_{\omega_1 \cdot k}$-free groups with trivial dual is provided for all integers $k\ge 1$, while
the nonexistence of $\aleph_{\omega_1 \cdot \omega}$-free groups with trivial dual is shown to be consistent with ZFC.

In this note we want to investigate the relation between $\kk$-free groups and cotorsion groups, where we call a group $C$
\emph{cotorsion} if $\Ext(F,C)=0$ for all torsion-free groups $F$. If $\mathfrak{F}$ and $\mathfrak{C}$ denote
the classes of torsion-free groups and cotorsion groups, respectively, then
\[\mathfrak{C}=\mathfrak{F}^\perp =\{ G \mid \Ext(F,G)=0 \mbox{ for all } F \in \mathfrak{F}\}\]
and
\[\mathfrak{F}= {}^\perp \mathfrak{C}=\{ G \mid \Ext(G,C)=0 \mbox{ for all } C \in \mathfrak{C}\}\]
holds, i.e., the pair of classes $(\mathfrak{F},\mathfrak{C})$ defines a \emph{cotorsion theory}. It should be
noted that a group $C$ is cotorsion if and only if  $\Ext(\Q,C)=0$ for the additive group of rationals $\Q$.
This is to say that $\Q$ is a \emph{cogenerator} of the cotorsion theory $(\mathfrak{F},\mathfrak{C})$. More generally,
we call a class $\mathfrak{F'} \subseteq \mathfrak{F}$ a \emph{cogenerating family} provided that any group $C$ is cotorsion
if and only if $\Ext(F,C)=0$ for all $F\in \mathfrak{F'}$. Therefore, $(\mathfrak{F},\mathfrak{C})$ is cogenerated by the
singleton $\{\Q\}$. Determining other cogenerating families for $(\mathfrak{F},\mathfrak{C})$ has been of interest and we note
in particular the following classical result~\cite{GP}.

\begin{theorem} \label{mainGP}
For any group $C$ the following statements are equivalent.
\begin{itemize}
\item[(i)] $C$ is cotorsion.
\item[(ii)] $\Ext(\Z^\lambda,C) = 0$ for some cardinal $\lambda$ with $\lambda^{\aleph_0}=2^\lambda \ge |C|$.
\end{itemize}
\end{theorem}

In particular, with $\l_0=|C|$ and $\l_{i+1} =2^{\l_i}$, the cardinal $\l= \bigcup_{i<\om} \l_i$ satisfies the property
$\lambda^{\aleph_0}=2^\lambda \ge |C|$, and the class of $\aleph_1$-free groups is a cogenerating family for $(\mathfrak{F},\mathfrak{C})$.
In this note we would like to add the class of $\aleph_k$-free groups $(k\ge 1)$ as yet another cogenerating family,
thus providing additional evidence that in ZFC the class of $\aleph_k$-free groups is large and of a rich structure.

\begin{theorem}[ZFC] \label{main}
Let $k\ge 1$ be some integer. Then the following statements are equivalent for any group $C$.
\begin{itemize}
\item[(i)] $C$ is cotorsion.
\item[(ii)] $\Ext(F,C) = 0$ for all $\aleph_k$-free groups $F$.
\end{itemize}
\end{theorem}

Notably, given any group $C$ that fails to be cotorsion, we will construct an $\aleph_k$-free group $F_C$ with
$\Ext(F_C,C) \ne 0$. To this end, Section \ref{2} provides an easy criterion for cotorsionness,
while Section \ref{3} reviews the $\bar\lambda$-Black Box. The final construction of $F_C$ is presented in Section \ref{4},
while Section \ref{5} provides an $\aleph_k$-free analog of Theorem~\ref{mainGP}.

It should be noted that the given argument easily adapts to other combinatorial principles, like the Jensen diamond
$\diamondsuit$, and we make a passing mention of the corresponding result.

\begin{corollary}[V=L]
Let $\kappa$ be some uncountable cardinal. Then the following statements are equivalent for any group $C$.
\begin{itemize}
\item[(i)] $C$ is cotorsion.
\item[(ii)] $\Ext(F,C) = 0$ for all $\kappa$-free groups $F$.
\end{itemize}
\end{corollary}

\section*{Acknowledgement}

We would like to thank Jan Trlifaj for bringing this problem to our attention.

\section{A characterization of cotorsion groups} \label{2}

The following criterion distinguishes between cotorsion groups and such groups that fail to be cotorsion in ways that
can be interpreted combinatorially. This will provide us later on with a useful foothold for applying the $\bar\lambda$-Black Box.

\begin{theorem} \label{cotcrit}
For any group $C$ the following statements are equivalent.
\begin{itemize}
\item[(i)] $\Ext(\Q, C) \not= 0$.
\item[(ii)] There exist elements $c_n \in C$ $(n\in \Z^{\ge 0})$
such that the infinite system of linear equations
\[x_n=(n+1)x_{n+1}+c_n\]
is not solvable in $C$.
\end{itemize}
\end{theorem}

\begin{proof}
For (i) implies (ii), let us consider some group $C$ with $\Ext(\Q, C) \not= 0$. Thus, there exists some short exact sequence
\[
\begin{tikzcd}
0 \ar[r] & C \ar[r] & G \ar[r,"\varphi"] & \Q \ar[r] & 0
\end{tikzcd}
\]
which fails to split. As usual, we will interpret $C$ as a subgroup of $G$. For $n\ge 0$ choose some $g_n\in G$ with
$\varphi(g_n)= \frac 1{n!}$. Then $\varphi(g_n)=\varphi((n+1)g_{n+1})$, and there exist $c_n \in C = \Ker \varphi$ with
\[g_n=(n+1)g_{n+1}+c_n.\]
We claim that the corresponding infinite system of equations
\[x_n=(n+1)x_{n+1}+c_n\]
has no solution in $C$. Towards a contradiction let us for the moment assume the existence of such a solution
$(x_n\mid n\in \Z^{\ge 0})$ with $x_n \in C \subseteq G$. Then $g_n-x_n \in G$ with $\varphi(g_n-x_n) = \varphi(g_n)= \frac 1{n!}$ and
\[g_n-x_n = (n+1)(g_{n+1}-x_{n+1}).\]
Thus, $\psi(\frac 1{n!}):=g_n-x_n$ defines a homomorphism $\psi: \Q \to G$ with $\varphi \circ \psi = \id_\Q$, and the
short exact sequence splits, contradicting our choice.

For (ii) implies (i), let $c_n \in C$ $(n\in \Z^{\ge 0})$ be a set of elements such that the corresponding system of equations
\[x_n=(n+1)x_{n+1}+c_n\]
is not solvable in $C$. For a set of free generators $y_n$ $(n\in \Z^{\ge 0})$, we define the groups
\[ U= \big\langle y_n-(n+1)y_{n+1}-c_n\mid n\ge 0 \big\rangle \subseteq C\oplus \bigoplus_{n\ge 0} \Z y_n \]
and
\[ V= \big\langle y_n-(n+1)y_{n+1}\mid n\ge 0 \big\rangle \subseteq \bigoplus_{n\ge 0} \Z y_n. \]
It is readily observed that $C$ embeds into $G :=\big(C\oplus \bigoplus_{n\ge 0} \Z y_n\big)/U$ canonically via $c\mapsto c+U$.
Furthermore, $H := \big(\bigoplus_{n\ge 0} \Z y_n\big)/V \cong \Q$, and the canonical projection
\[\pi: C\oplus \bigoplus_{n\ge 0} \Z y_n \to \bigoplus_{n\ge 0} \Z y_n\]
induces a homomorphism $\bar\pi: G\to H$ with $\bar\pi(y_n+U)=y_n+V$ and $\bar\pi(c+U)=0$.
Using the fact that every element of $G$ can be represented in the form
$(c+zy_m)+U$ for suitable $c\in C$, $z\in \Z$, and $m\ge 0$, we can check $\Ker \bar\pi =C$. Summarizing, we have the short
exact sequence
\[
\begin{tikzcd}
0 \ar[r] & C \ar[r] & G \ar[r,"\bar\pi"] & H\cong \Q \ar[r] & 0,
\end{tikzcd}
\]
and we claim that this exact sequence does not split. Towards a contradiction let us for the moment assume the existence
of a splitting homomorphism $\psi: H \to G$ with $\bar\pi \circ \psi = \id_H$. We then have
\[(y_n+U) - \psi(y_n+V) \in \Ker \bar\pi =C,\]
and with $x_n := (y_n+U) - \psi(y_n+V) \in C$ holds
\begin{eqnarray*}
x_n-(n+1)x_{n+1} & = & \big(y_n-(n+1)y_{n+1}+U\big) - \psi\big(y_n-(n+1)y_{n+1}+V\big)\\
& = & (c_n+U) - \psi(0+V) = c_n+U
\end{eqnarray*}
in $G$. From this we infer $x_n=(n+1)x_{n+1}+c_n$ in $C\subseteq G$, contradicting (ii). Hence, the aforementioned
exact sequence does not split, and $\Ext(\Q, C) \not= 0$ follows.
\end{proof}

\section{The $\bar\lambda$-Black Box} \label{3}

We recall the basics of the $\bar\lambda$-Black Box, keeping this exposition rather short with the
intention of providing a fast and simple reference for future $\aleph_k$-free constructions in ZFC.
The proofs of Lemma \ref{freeprop} and Theorem \ref{freethm} can be skipped for faster access.
The reader may consult \cite{GS3,Habil,HSa} for further details and any left out proofs.

\subsection{$\Lambda$ and $\Lambda_*$} \label{3.1}

Throughout this section, we will employ some standard notations from set theory. In particular,
we will identify $0=\emptyset$, $n =\{0,\ldots,n-1\}$ for every positive integer $n$, and
$\alpha =\{\beta \mid \beta < \alpha \}$ for every ordinal $\alpha$. Let
$\omega =\{0,1,2,\ldots\}$ denote the first infinite ordinal. Ordinals will be assigned letters
$\alpha$, $\beta$, while cardinals will be assigned letters $\kappa$, $\lambda$.

\begin{notation}
Let $^\omega\lambda$ denote the set of all functions $\tau:\omega\to\lambda$, while $^{\omega\uparrow}\lambda$ is the subset of $^\omega\lambda$ consisting of all strictly increasing
functions $\eta:\omega\to\lambda$, namely
     \[
     ^{\omega\uparrow}\lambda=\{\eta:\omega\to\lambda \mid \eta(m)< \eta(n) \text{ for all } m<n\}.
     \]
     Similarly, $^{\omega>}\lambda$ denotes the set of all functions $\sigma:n\to\lambda$ with $n<\omega$, while $^{\omega\uparrow>}\lambda$ is the subset of $^{\omega>}\lambda$ consisting of
     all strictly increasing functions $\eta:n\to\lambda$ with $n<\omega$.
\end{notation}

For some integer $k\ge 1$, let  $\bar\lambda=\langle \lambda_1,\ldots,\lambda_{k}\rangle$ be a finite increasing sequence of infinite cardinals with the following properties:
\begin{enumerate}
\item[(i)] $\lambda^{\aleph_0}_1=\lambda_1$.
\item[(ii)] $\lambda^{\lambda_m}_{m+1}=\lambda_{m+1}$ for all $1\le m < k$.
\end{enumerate}
In particular, the sequence $\bar\lambda=\langle \beth_1,\ldots, \beth_k\rangle$ is an example and constitutes the smallest possible choice for $\bar\lambda$.

We associate with $\bar\lambda$ two sets
$\Lambda$ and $\Lambda_*$. Let
\[\Lambda= \lup_1\times \ldots \times \lup_k.\]
For the second set we replace the $m$-th (and only the $m$-th)
coordinate $\lup_m$ by $\lupf_m$, thus let
\[ \Lambda_{m*} = \lup_1\times \ldots \times   \lupf_m \times \ldots \times  \lup_{k} \text{ for } 1\le m\le k \text{ and }
\Lambda_*=\bigcup_{1\le m\le k}\Lambda_{m*}.\]

The elements of $\Lambda, \Lambda_*$ will be written as sequences $\etabar=(\eta_1,\dots,\eta_k)$ with $\eta_m\in \lup_m$ or $\eta_m \in \lupf_m$, respectively.
With each member of $\etabar\in \Lambda$ we associate some elements of $\Lambda_*$ which result from restricting the length of one of the entries $\eta_m \in \lup$ of $\etabar$.

\begin{definition}\label{support}
If $\etabar = (\eta_1,\dots,\eta_k)\in \Lambda$ and $1\le m\leq k,n<\om$, then let
$\etabar\restl\langle m,n\rangle$ be the following element of
$\Lambda_{m*} \subseteq\Lambda_*$
\[
(\etabar\restl\langle m,n\rangle)_l =
\begin{cases}
 \ \eta_l & \text{ if }\  m\ne l \le k,\\
 \ \eta_m\restr n & \text{ if }\ l = m.
\end{cases}
\]
We associate with $\etabar$ its {\em support}
\[ [\etabar] =\{\etabar\restl\langle m,n\rangle \mid1\le  1\le m\le k, n < \om\}\] which is a
countable subset of $\Lambda_*$.
\end{definition}

\subsection{The modules} \label{3.2}

Let $R$ be a commutative ring with $1$ and let $\bS \subseteq R\setminus \{0\}$ be a countable multiplicatively closed subset. We introduce the following basic concepts.

\begin{definition} \mbox{}
\begin{enumerate}
\item[(a)] An $R$-module $M$ is \emph{$\mathbb{S}$-torsion-free} if $sm = 0$ for $s\in\mathbb{S}$, $m\in M$ implies $m = 0$.
\item[(b)] An $R$-module $M$ is \emph{$\mathbb{S}$-reduced} if $\bigcap_{s\in\mathbb{S}}sM=0$.
\item[(c)] The ring $R$ is an \emph{$\mathbb{S}$-ring} if $R$ as an $R$-module is $\mathbb{S}$-torsion-free and $\mathbb{S}$-reduced.
\item[(d)] Let $M$ be an $R$-module. A submodule $N\subseteq M$ is \emph{$\mathbb{S}$-pure} if $N\cap sM=sN$ for all $s\in\mathbb{S}$. We write $N\subseteq_* M$.
\item[(e)] Let $M$ be an $\mathbb{S}$-torsion-free $R$-module, and let $T$ be a subset of $M$. Then $\langle T\rangle_*$ will denote the smallest $\mathbb{S}$-pure
submodule of $M$ containing $T$.
\end{enumerate}
\end{definition}

In the following, $R$ will always denote an $\mathbb{S}$-ring. Furthermore, we enumerate $\bS=\{s_i \mid i< \omega\}$
and put $q_n =\prod_{i< n}s_i$; thus, $q_0=1$ and $q_{n+1} = q_n s_{n}$.
The $\bS$-topology on $R$, generated by the basis $sR$ $(s\in \bS)$ of neighbourhoods of $0$, is Hausdorff
and we can consider the $\bS$-completion $\Rhat$ of $R$. Note $R\subseteq_* \Rhat$, and see \cite{GT} for further basic facts on $\Rhat$.

\begin{remark}
The case $R =\Z$ presents us with two canonical options for $\bS$.
\begin{enumerate}
\item [(i)]  For any prime $p$, the choice $\bS=\{p^i | i\in \Z^{\ge 0}\}$ gives the $p$-adic topology.
\item [(ii)] The choice $\bS=\Z^{>0}$ gives the $\Z$-adic topology.
\end{enumerate}
\end{remark}

The choice of $R$-modules is the most flexible part of the $\bar\lambda$-Black Box and very much depends on the respective goals of the final construction.
Here we will present only one simple generic example to discuss some of the more common features of $\bar\lambda$-Black Box constructions. In particular, it should
be noted that the following general statement will be responsible for $\aleph_k$-freeness of the constructed $R$-modules, where $\mathcal{P}^{\operatorname{fin}}(T)$
denotes the set of all finite subsets of a given set $T$.

\begin{lemma}[{\cite[Proposition 3.5]{Habil}}]\label{freeprop}
Let $F: \Lambda \to \mathcal{P}^{\operatorname{fin}}(\Lambda_*)$ be any function, $1\le f\le k$ and
$\Omega$ a subset of $\Lambda$ of cardinality $\aleph_{f-1}$ with a family of sets
$u_\etabar\subseteq \{1,\dots,k\}$ satisfying $\size{u_\etabar}\ge f$ for all $\etabar\in \Omega$.
Then we can find an enumeration $\langle
\etabar^\a\mid \a<\aleph_{f-1}\rangle$  of $\Omega$, $\ell_\a \in
u_{\etabar^\a}$ and $ n_\a <\omega$ $(\a<\aleph_{f-1})$ such that
\[\etabar^\a\restl \langle \ell_\a,n\rangle\notin \{\etabar^\b\restl \langle \ell_\a,n\rangle \mid\b< \a\}
\cup \bigcup \big\{ F(\etabar^\b) \mid \b\le \alpha\big\} \text{ for all } n\ge n_\a.\]
\end{lemma}

\begin{proof} The proof follows by induction on $f$. We begin with $f=1$,
so $\size{\Omega}=\aln$. Let $\Omega
=\{\etabar^\a\mid \a <\omega\}$ be any enumeration without repetitions.
From $1=f\le\size{u_\etabar}$ follows $u_\etabar\ne \emptyset$
and we choose any $\ell_\a\in  u_{{\etabar}^\a}$ for $\a
<\omega$. If $\a\ne\b <\omega$, then $\etabar^\a \ne \etabar^\b $ and there is
$n_{\a \b} \in \omega$ such that $\etabar^\a\restl \langle
\ell_\a,n\rangle \ne\etabar^\b\restl \langle
\ell_\a,n\rangle$ for all $n\ge n_{\a\b}$.
Since $\bigcup \{ F(\etabar^\b) \mid \b\le \alpha\}$ is finite, we may enlarge $n_{\a \b}$, if necessary, such
that $\etabar^\a \restl \langle \ell_\a,n\rangle \notin \bigcup \{ F(\etabar^\b) \mid \b\le \alpha\}$
for all $n\ge n_{\a \b}$. If $n_\a=\max_{\b <\a}n_{\a \b}$, then
\[\etabar^\a\restl \langle \ell_\a,n\rangle\notin\{\etabar^\b\restl \langle \ell_\a,n\rangle \mid\b< \a\}
\cup \bigcup \{ F(\etabar^\b) \mid \b\le \alpha\} \text{ for all } n\ge n_\a.\]
Hence the case $f=1$ is settled. For the induction step, we let $f'= f+1$ and assume
that the lemma holds for $f$.

Let $\size{\Omega} =\aleph_{f}$ and choose an
${\aleph_{f}}$-filtration $\Omega
=\bigcup_{\d<\aleph_{f}}\Omega_\delta$ with $\Omega_0=\emptyset$ and
$\size{\Omega_{\d+1}\setminus \Omega_\d}=\aleph_{f-1}$ $(\d < \aleph_{f})$. The next crucial idea comes from
\cite{S3}: We can also assume that the chain $\{\Omega_\d\mid \d <\aleph_f\}$ is
\emph{closed}, meaning that for any $\d <\aleph_{f}$, $\nubar,\nubar'\in
\Omega_\delta$ and $\etabar\in \Omega$ with
\[\{\eta_m\mid 1\le m\le k\}\subseteq \{\nu_m,\nu'_m,\nu''_m \mid  \nubar'' \in F(\nubar) \cup F(\nubar'), 1\le m\le k\}\]
follows $\etabar\in \Omega_\delta$. Thus, if $\etabar\in \Omega_{\delta+1}\setminus \Omega_\delta$, then the set
\[u^*_\etabar =\{1\le \ell\le k\mid \exists n<\omega,\nubar\in \Omega_\d \text{ such that } \etabar\restl \langle \ell, n\rangle =\nubar\restl
\langle \ell, n\rangle \text{ or } \etabar\restl \langle\ell,n\rangle \in F(\nubar)\}\]
is empty or a singleton. Otherwise there are  $n,n'
<\omega$ and distinct $1\le \ell, \ell'\le k$ with $\etabar\restl
\langle \ell, n\rangle \in \{ \nubar\restl \langle \ell, n\rangle\}
\cup  F(\nubar)$ and $\etabar\restl \langle \ell', n'\rangle \in
\{\nubar'\restl \langle \ell', n'\rangle\}\cup  F(\nubar')$ for
certain $\nubar,\nubar'\in \Omega_\d$. Hence
\[\{\eta_m\mid 1\le m\le k\}\subseteq \{\nu_m,\nu'_m, \nu''_m \mid \nu''_m\in F(\nubar) \cup F(\nubar'),  1\le m\le k\},\]
and the closure property implies the contradiction $\etabar\in \Omega_\delta$.

If $\d <\aleph_{f}$, then let $D_\d=\Omega_{\d+1}\setminus \Omega_\d$
with $\size{D_\d}=\aleph_{f-1}$, and $u'_\etabar: = u_\etabar\setminus u^*_\etabar$ must have size
$\ge f'-1 =f$. Thus, the induction hypothesis applies to $\{u'_\etabar\mid \etabar \in D_\delta\}$
for each $\d <\aleph_f$ and we
find an enumeration $\langle \etabar^{\d\a} \mid \a < \aleph_{f-1}\rangle$ of $ D_\d$
as in the lemma. Finally, putting for $\d<\aleph_{f}$ all these enumerations together with the standard induced
ordering, we find an enumeration $\langle \etabar^\a\mid \a<\aleph_{f}\rangle$ of $\Omega$ satisfying the lemma.
\end{proof}

The sets $u_\etabar$ in Lemma \ref{freeprop} are merely auxiliary for the induction proof
and one may rather want to focus oneself on the following simplified statement.

\begin{theorem}\label{freelem}
For any function $F: \Lambda \to \mathcal{P}^{\operatorname{fin}}(\Lambda_*)$, and any subset of $\Omega$ of $\Lambda$ of cardinality $|\Omega|<\aleph_k$,
we can find an enumeration $\langle \etabar^\a\mid \a<|\Omega|\rangle$ of $\Omega$, and elements $1\le \ell_\a \le k$ and $n_\a<\om$ $(\a < |\Omega|)$ such that
\[\etabar^\a\restl \langle \ell_\a,n\rangle\notin \{\etabar^\b\restl \langle \ell_\a,n\rangle \mid\b< \a\}
\cup \bigcup \big\{ F(\etabar^\b) \mid \b\le \alpha\big\} \text{ for all } n\ge n_\a.\]
\end{theorem}

\begin{remark}
In other words, every element $\etabar^\a$ of this enumeration picks up some new element from $\Lambda_*$ in its support $[\etabar^\a]$ which has not been
associated with any of the previous elements $\etabar^\b$ $(\b<\a)$. This will be the core of the support argument in the proof of Theorem \ref{freethm}.
\end{remark}

We continue with a description of the most common setup for $\aleph_k$-free constructions in ZFC.
We start with the $R$-module
\[B=\bigoplus_{\nubar\in \Lambda_*} Re_{\nubar}\]
freely generated by $\{e_{\nubar}\mid \nubar\in \Lambda_*\}$ over the $\mathbb{S}$-ring $R$. The $\mathbb{S}$-topology of $R$ naturally extends to the
$\mathbb{S}$-topology of $B$ generated by the basis $sB$ $(s\in \mathbb{S})$ of neighborhoods of $0$. Let
\[\Bhat \subseteq \prod_{\nubar\in \Lambda_*} \Rhat e_{\nubar}\]
denote the $\mathbb{S}$-completion of $B$. Thus every element $b\in \Bhat$ can be written canonically as a sum $b=\sum_{\nubar\in \Lambda_*}b_\nubar e_\nubar$
with coefficients $b_\nubar\in \Rhat$, and
\[[b]=\{\nubar\in \Lambda_*\mid b_\nubar\ne 0\}\]
will denote the \emph{support} of $b$. We have $B\subseteq_* \Bhat$, and we intend to construct an $\aleph_k$-free module
\[B\subseteq_* M\subseteq_* \Bhat\]
by adding suitable elements $y'_\etabar \in \Bhat$ $(\etabar \in \Lambda)$ to $B$.

For $\etabar\in \Lambda$ and $i<\omega$, we call
\[ y_{\etabar i}=\sum_{n= i}^\infty \frac{q_n}{q_i}\Bigg(\sum_{m=1}^{k}e_{\etabar\restl \langle m,n\rangle}\Bigg)\]
the \emph{branch element} associated with $\etabar$. In particular, let
\[y_{\etabar} = y_{\etabar 0} =\sum_{n=0}^{\infty}q_n \Bigg(\sum_{m=1}^{k}e_{\etabar\restl \langle m,n\rangle}\Bigg).\]

In addition, given a function $F: \Lambda \to \mathcal{P}^{\operatorname{fin}}(\Lambda_*)$ we choose elements $b_{\etabar n}\in B$
for $\etabar\in \Lambda$ and $n<\omega$ with $[b_{\etabar n}] \subseteq F(\etabar)$. Then we
introduce {\em branch-like elements} $y'_{\etabar i}$ by adding some corrections to our branch-elements $y_{\etabar i}$, namely
\[ y'_{\etabar i}= \sum_{n= i}^\infty \frac{q_n}{q_i}\Bigg(b_{\etabar n} +\sum_{m=1}^{k}e_{\etabar\restl \langle m,n\rangle}\Bigg)
= y_{\etabar i} +\sum_{n= i}^\infty \frac{q_n}{q_i} b_{\etabar n}.\]
In particular, we have
\[ y'_{\etabar}=y'_{\etabar 0}= \sum_{n= 0}^\infty q_n \Bigg(b_{\etabar n} +\sum_{m=1}^{k}e_{\etabar\restl \langle m,n\rangle}\Bigg)
= y_{\etabar} +\sum_{n= 0}^\infty q_n b_{\etabar n}.\]
Note $[y_\etabar] =[\etabar]$ and $[y'_\etabar] \subseteq F(\etabar) \cup [\etabar]$. Our module of interest is now given by
\[M= \langle B, y'_\etabar \mid \etabar \in \Lambda \rangle_* \subseteq_* \Bhat.\]
Note the following helpful recursions
\begin{align}\label{eq1}
y_{\etabar i}= s_i y_{\etabar, i+1}+\sum_{m=1}^{k}e_{\etabar\restl \langle m,i\rangle} \quad \mbox{ and } \quad
y'_{\etabar i}= s_i y'_{\etabar, i+1}+b_{\etabar i} +\sum_{m=1}^{k}e_{\etabar\restl \langle m,i\rangle}.
\end{align}
As a consequence we have the identity
\[M= \langle B, y'_\etabar \mid \etabar \in \Lambda \rangle_* =\langle B, y'_{\etabar i} \mid \etabar \in \Lambda, i < \omega \rangle.\]
The central theorem of this section is now the following statement about $\aleph_k$-freeness.

\begin{theorem}\label{freethm}
Let $M$ be the $R$-module
\[M=\langle B, y'_{\etabar i} \mid \etabar \in \Lambda, i < \omega \rangle = \langle B, y'_\etabar \mid \etabar \in \Lambda \rangle_* \subseteq_* \Bhat.\]
Then any subset $T$ of $M$ with $|T|<\aleph_k$ is contained in a free submodule $N\subseteq M$.
\end{theorem}

\begin{proof}
With $M=\langle B, y'_{\etabar i} \mid \etabar \in \Lambda, i < \omega \rangle$, every element $g\in M$ can be written as an $R$-linear combination
of finitely many branch-like elements $y'_{\etabar i}$ and of finitely many generators $e_{\etabar\restl \langle m,n\rangle}$ of $B$. In particular,
collecting all $y'_{\etabar i}$ and $e_{\etabar\restl \langle m,n\rangle}$ needed for representing the elements $g \in T$, there exists a subset
$\Omega$ of $\Lambda$ of size $|\Omega| <\aleph_k$ such that $T$ is a subset of the submodule
\[M_\Omega=\langle e_{\etabar\restl \langle m,n\rangle},e_\nubar, y'_{\etabar n} \mid \etabar \in \Omega, \nubar \in F(\etabar), 1\le m\le k, n<\om \rangle \subseteq M.\]
To complete the proof, we will show that $M_\Omega$ is a free $R$-module.

With Theorem \ref{freelem} we write
\[M_\Omega =\langle  e_{\etabar^\a\restl \langle m,n\rangle}, e_\nubar, y'_{\etabar^\a n} \mid \a <|\Omega|, \nubar \in F(\etabar^\a), 1\le m\le k, n<\om \rangle\]
for a list $\langle \etabar^\a\mid \a<|\Omega|\rangle$ of $\Omega$ for which there exist $1\le \ell_\a \le k$ and $n_\a < \om$ with
\begin{align}\label{eq2}
\etabar^\a\restl \langle \ell_\a,n\rangle\notin \{\etabar^\b\restl \langle \ell_\a,n\rangle \mid\b< \a\} \cup \bigcup \big\{ F(\etabar^\b) \mid \b\le \alpha\big\}
\end{align}
for all $n\ge n_\a$. Let
\[M_\a =\langle  e_{\etabar^\g\restl \langle m,n\rangle}, e_\nubar, y'_{\etabar^\g n} \mid \g <\a, \nubar \in F(\etabar^\g), 1\le m\le k, n<\om \rangle\]
for any $\a < |\Omega|$. With \eqref{eq1}, we have
\begin{eqnarray*}
M_{\a +1} & = & M_\a+\langle e_{\etabar^\a\restl \langle m,n\rangle}, e_\nubar, y'_{\etabar^\a n} \mid \nubar \in F(\etabar^\a), 1\le m\le k, n<\om \rangle\\
& = & M_\a  + \langle y'_{\etabar^\a n}\mid n \ge n_\a\rangle +\langle e_{\etabar^\a\restl \langle \ell_\a,n\rangle}\mid n < n_\a\rangle\\
& & \quad\;\;\,  +\; \langle e_{\nubar}, e_{\etabar^\a\restl \langle m,n\rangle} \mid \nubar \in F(\etabar^\a), 1 \le m \le k, m\ne \ell_\a, n<\om \rangle.
\end{eqnarray*}

Hence, any element in $M_{\a +1}$ can be represented as a sum of the form
\[ g+\sum_{n\ge n_\a} r_n y'_{\etabar^\a n}+ \sum_{n<n_\a}r'_n e_{\etabar^\a\restl \langle
\ell_\a,n\rangle} + \sum_{\nubar \in F(\etabar^\a)} r_\nubar e_\nubar
+\sum_{n < \om} \sum_{\stackrel{1 \le m\le k}{m\ \ne\ \ell_\a}}r''_{mn} e_{\etabar^\a\restl \langle m,n\rangle},\]
where $g \in M_\a$, and all coefficients $r_n,r'_n, r_\nubar, r''_{mn}$ are from $R$. Moreover,
identifying $e_\nubar \ (\nubar \in F(\etabar^\a))$ with one of the $e_{\etabar^\a\restl \langle m,n \rangle}$s
whenever possible and merging all $e_{\etabar^\alpha}\restl \langle m,n\rangle, e_\nubar \in M_\a$ into $g$, we may
slightly simplify this sum.

Assume now that the above sum is zero. Condition \eqref{eq2} implies that $e_{\etabar^\a\restl \langle \ell_\a,n\rangle}$
contributes in this sum only to the branch part $y_{\etabar^\a n'}$ of $y'_{\etabar^\a n'}$ for $n_\a \le n' \le n$. Applying this to the $y'_{\etabar^\a n}$s,
starting with the smallest appearing $n$, we have $r'_n=0$ for all $n\ge n_\a$. Moreover, the remaining summands $g$,
$e_{\etabar^\a\restl \langle m,n\rangle}$, and $e_\nubar$ trivially have disjoint supports. Thus, also all the
coefficients $r_n, r_{\nubar},r''_{mn}$, and consequently also $g$ must be zero.
This shows that $M_{\a +1}=M_\a \oplus \bigoplus_{b \in \mathcal{B}_\alpha} Rb$ for
\begin{eqnarray*}
\mathcal{B}_\alpha & = & \{ y'_{\etabar^\a i}, e_{\etabar^\a\restl \langle \ell_\a,j\rangle},  e_{\etabar^\a\restl \langle m,n\rangle}, e_\nubar \mid\\
& & \qquad i \ge n_\a, j < n_\a, 1\le m \le k, m \ne \ell_\a, n<\omega, \nubar\in F(\etabar^\alpha)\} \setminus M_\a,
\end{eqnarray*}
and $M_\Omega=\bigoplus_{\alpha < |\Omega|} \bigoplus_{b \in \mathcal{B}_\alpha} Rb$
is a free $R$-module.
\end{proof}

\begin{remark}
It should be noted that the statements of Lemma \ref{freeprop}, Theorem \ref{freelem} and Theorem \ref{freethm} hold for any choice of infinite cardinals
$\lambda_1,\ldots,\lambda_k$. The additional properties of  $\bar\lambda$ required in Section~\ref{3.1} are irrelevant for the $\aleph_k$-freeness
and are only needed to obtain the added prediction feature. This will be our next stop!
\end{remark}

\subsection{The prediction}

No black box would be complete without some prediction principle, and it is noteworthy that the prediction
of any black box can be traced back to the following simple general statement.

\begin{theorem}[The Easy Black Box] Let $\l$ be an infinite cardinal and let $C$ be a set of size $|C|\le 2^\l$. Then
there exists some family $\langle \varphi_\eta \mid \eta \in \lup\rangle$ of functions $\varphi_\eta: \om \to C$
such that the following holds.\medskip

{\sc prediction principle:} Given any map $\varphi: \lupf \to C$ and any ordinal $\alpha \in \l$, there exists
some $\eta \in \lup$ such that $\eta(0)=\alpha$ and $\varphi_\eta(n)= \varphi(\eta\restr n)$ for all $n < \om$.
\end{theorem}

In particular, the $\bar\l$-Black Box for $\bar\lambda=\langle \lambda_1,\ldots,\lambda_{k}\rangle$
basically constitutes the result of stacking $k$ Easy Black Boxes on top of each other.

\begin{theorem}[The $\bar\l$-Black Box]\label{bb}
For $\bar\lambda=\langle \lambda_1,\ldots,\lambda_{k}\rangle$ a sequence of cardinals as in Section \ref{3.1},
let $\bar C= \langle C_1,\ldots,C_{k}\rangle$ be a sequence with $|C_m|\leq\lambda_m$ $(1\le m\le k)$, and let
$C=\bigcup_{1\le m\le k}C_m$. Then
there exists some family $\langle \va_\etabar \mid \etabar\in \Lambda\rangle$
of functions $\varphi_{\bar\eta}:[\bar\eta]\to C$ such that the following holds.\medskip

{\sc prediction principle:} Given any map $\va: \Lambda_* \to C$ with $\Lambda_m\va \subseteq C_m$  for all
$1\le m\le k$, and given any ordinal $\alpha \in \l$, there exists
some $\etabar\in\Lambda$ such that $\eta_k(0)=\alpha$ and $ \va_\etabar\subseteq \va$.
\end{theorem}

\section{The proof of Theorem \ref{main}} \label{4}

For the proof of Theorem \ref{main}, (i) obviously implies (ii) as all $\aleph_k$-free groups are torsion-free.
Thus, we only need to verify the converse statement. To that effect, we will start with a group $C$ that fails to be
cotorsion, and we must provide an $\aleph_k$-free group $F_C$ with $\Ext(F_C,C) \ne 0$.

As $C$ fails to be cotorsion, with Theorem \ref{cotcrit}, we choose elements $c_n \in C$ $(n<\om)$
such that the infinite system of linear equations
\begin{align}\label{con0}
x_n=(n+1)x_{n+1}+c_n
\end{align}
is not solvable in $C$. For an infinite cardinal $\kappa \ge |C|$, let
\[\l_1= \kappa^{\aleph_0} \ge |C| \quad \mbox{ and } \quad \l_{i+1} =2^{\l_i}.\]
Then $\bar\lambda=\langle \lambda_1,\ldots,\lambda_{k}\rangle$ satisfies the properties of Section \ref{3.1}, and
we will use the prediction of the $\bar\lambda$-Black Box for the choice $C_m = C$ $(1\le m\le k)$, cf.~Theorem \ref{bb}.
In particular, there exists some family $\langle \va_\etabar \mid \etabar\in \Lambda\rangle$
of functions $\varphi_{\bar\eta}:[\bar\eta]\to C$ such that the following prediction principle holds.\medskip
\begin{align} \label{con1}
\mbox{Given any map $\va: \Lambda_* \to C$, there exists some $\etabar\in\Lambda$ such that $ \va_\etabar\subseteq \va$.}
\end{align}
We next want to construct two groups $F_C$ and $G_C$. To start with, let
\[B=\bigoplus_{\nubar\in \Lambda_*} \Z e_{\nubar}\]
be the group freely generated by $\{e_{\nubar}\mid \nubar\in \Lambda_*\}$.
Let $\Zhat$ and $\Bhat$ denote the $\mathbb{Z}$-adic completions of $\Z$ and $B$, respectively.
Every element $b\in \Bhat$ can be written canonically as a sum $b=\sum_{\nubar\in \Lambda_*}b_\nubar e_\nubar$
with coefficients $b_\nubar\in \Zhat$, and
\[[b]=\{\nubar\in \Lambda_*\mid b_\nubar\ne 0\}\]
will denote the \emph{support} of $b$. For $\etabar\in \Lambda$ and $i<\omega$, we call
\[ y_{\etabar i}=\sum_{n= i}^\infty \frac{n!}{i!}\Bigg(\sum_{m=1}^{k}e_{\etabar\restl \langle m,n\rangle}\Bigg)\]
the \emph{branch element} associated with $\etabar$. In particular, let
\[y_{\etabar} = y_{\etabar 0} =\sum_{n=0}^{\infty} n! \Bigg(\sum_{m=1}^{k}e_{\etabar\restl \langle m,n\rangle}\Bigg).\]
Note again the recursion
\begin{align}\label{con2}
y_{\etabar i}= (i+1) y_{\etabar, i+1}+\sum_{m=1}^{k}e_{\etabar\restl \langle m,i\rangle}.
\end{align}
These formulas are identical to those in Section \ref{3.2} for the choice $\mathbb{S}=\Z^{>0}$ and $s_i=i+1$.
We now define
\[F_C=\langle B, y_{\etabar i} \mid \etabar \in \Lambda, i < \omega \rangle = \langle B, y_\etabar \mid \etabar \in \Lambda \rangle_* \subseteq_* \Bhat.\]

\begin{lemma}
The group $F_C$ is $\aleph_k$-free.
\end{lemma}

\begin{proof}
Let $H\subseteq F_C$ be a subgroup of cardinality $|H|<\aleph_k$. Then, with Theorem~\ref{freethm}, $H$ is contained in a free subgroup of $F_C$
and therefore free itself.
\end{proof}

We start our construction of the group $G_C$ with a little auxiliary gimmick to overcome $C$ not embedding into its $\Z$-adic completion $\widehat C$, as
$\bigcap_{n \in  \Z^{> 0}} nC \ne 0$ may quite be possible.
Let $C^\om =\prod_{n<\om} C e_n$ denote the cartesian product of countably infinitely many copies of $C$. Every element $g\in C^\om$ can be written
canonically as a sum $g=\sum_{n=0}^\infty g_n e_n$ with coefficients $g_n \in C$, and $[g]=\{n<\om \mid g_n \ne 0\}$ will denote the support of $g$. We define
the groups
\[C^\om_{\operatorname{fin}} =\bigg\{g \in C^\om \bigg|\; [g] \mbox{ is finite with } \sum_{n=0}^\infty g_n =0 \bigg\} \subseteq C^\om\]
and
\[\bar C = C^\om/ C^\om_{\operatorname{fin}}.\]
Note that $C$ canonically embeds into $\bar C$ via $c \mapsto ce_0 + C^\om_{\operatorname{fin}} = ce_n + C^\om_{\operatorname{fin}}$.

The group $G_C$ will be constructed as a subgroup of $\Bhat \oplus \bar C$ and will incorporate our $\bar\lambda$-Black Box predictions $\varphi_\etabar$
$(\etabar \in \Lambda)$ and the preselected elements $c_n \in C$ $(n<\om)$.
For $\etabar\in \Lambda$ and $i<\omega$, let
\[ z_{\etabar i}=y_{\etabar i} + \Bigg(\sum_{n= i}^\infty \frac{n!}{i!}\bigg(c_n-\sum_{m=1}^{k}\varphi_\etabar\big(\etabar\restl \langle m,n\rangle\big)\bigg)e_n + C^\om_{\operatorname{fin}}\Bigg) \subseteq \Bhat \oplus \bar C.\]
Again we have a recursion
\begin{align}\label{con3}
z_{\etabar i}= (i+1) z_{\etabar, i+1}+\sum_{m=1}^{k}e_{\etabar\restl \langle m,i\rangle}+
\Bigg( \bigg(c_i-\sum_{m=1}^{k}\varphi_\etabar\big(\etabar\restl \langle m,i\rangle\big)\bigg)e_i + C^\om_{\operatorname{fin}}\Bigg).
\end{align}
We now define
\[G_C=\langle B\oplus C, z_{\etabar i} \mid \etabar \in \Lambda, i < \omega \rangle \subseteq \Bhat \oplus \bar C.\]

Let $\pi: \Bhat \oplus \bar C \to \Bhat$ denote the canonical projection. Then $\pi(e_\nubar)=e_\nubar$ for all $\nubar \in \Lambda_*$ and
$\pi(z_{\etabar i})=y_{\etabar i}$ for all $\etabar \in \Lambda, i<\om$, thus $\pi(G_C)=F_C$.

\begin{lemma} \label{lem1}
We have $G_C \cap \Ker \pi =C$.
\end{lemma}

\begin{proof}
Every element $g\in G_C$ can be written as a linear combination of some element from $B\oplus C$ with finitely many elements $z_{\etabar i}$.
With \eqref{con3} we can limit this representation to one element $z_{\etabar i}$ for each $\etabar \in \Lambda$. Thus, we can write
\[g = b+ (ce_0+ C^\om_{\operatorname{fin}})+\sum_{\a = 0}^N n_\a z_{\etabar^\a i^\a},\]
where $b\in B$, $c\in C$, $N \in \Z^{\ge 0}$, and  $n^\a \in \Z$, $i^\a \in \Z^{\ge 0}$ for all $0\le \a \le N$ with distinct $\etabar^\a \in \Lambda$.
Let us assume $\pi(g)=0$.

Applying Theorem \ref{freelem} for the function $F: \Lambda \to \mathcal{P}^{\operatorname{fin}}(\Lambda_*)$ with $F(\nubar) = [b]$ constant for $\nubar \in \Lambda_*$,
we may assume that every element $\etabar^\a$ of the enumeration $\langle \etabar^\a \mid 0\le \a \le N \rangle$ picks up some new element from $\Lambda_*$ in its
support $[\etabar^\a]$ which has not been associated with $b$ or any of the previous elements $\etabar^\b$ $(\b<\a)$. Thus, $\pi(g)=0$ implies $n_\a =0$ for all
$0\le \a \le N$, and $g=b+ (ce_0+ C^\om_{\operatorname{fin}})$. Hence, $\pi(g)=b=0$ implies  $g=ce_0+ C^\om_{\operatorname{fin}} \in C$.
\end{proof}

The following lemma completes the proof of Theorem \ref{main}.

\begin{lemma}
We have $\Ext(F_C,C)\ne 0$.
\end{lemma}

\begin{proof}
With Lemma \ref{lem1}, we have the short exact sequence
\[
\begin{tikzcd}
0 \ar[r] & C \ar[r] & G_C \ar[r,"\pi"] & F_C \ar[r] & 0,
\end{tikzcd}
\]
and we claim that this exact sequence does not split. Towards a contradiction let us for the moment assume the existence
of a splitting homomorphism $\psi: F_C \to G_C$ with $\pi \circ \psi = \id_{F_C}$. Then in $G_C$ holds
\[e_\nubar - \psi(e_\nubar) \in G_C \cap \Ker \pi =C\]
for all $\nubar \in \Lambda_*$, and we define the function $\va: \Lambda_* \to C$ by $\va(\nubar) =e_\nubar - \psi(e_\nubar)$.
With~\eqref{con1}, we can choose some $\etabar \in \Lambda$ such that $\va_\etabar\subseteq \va$, thus
\[\va(\etabar\restl \langle m,n\rangle)=\va_\etabar(\etabar\restl \langle m,n\rangle) \]
for all $1\le m \le k$ and $n<\om$. In $G_C$ holds
\[z_{\etabar n} - \psi(y_{\etabar n}) \in G_C \cap \Ker \pi =C,\]
and we set $x_n := z_{\etabar n} - \psi(y_{\etabar n}) \in C$ for $n<\om$. With \eqref{con2} and \eqref{con3} we then have
\begin{eqnarray*}
& & x_n-(n+1)x_{n+1} = \big(z_{\etabar n}-(n+1)z_{\etabar, n+1}\big) - \psi\big(y_{\etabar n}-(n+1)y_{\etabar, n+1}\big)\\
& = & \sum_{m=1}^{k}e_{\etabar\restl \langle m,n\rangle}+
\Bigg( \bigg(c_n-\sum_{m=1}^{k}\varphi_\etabar\big(\etabar\restl \langle m,n\rangle\big)\bigg)e_n + C^\om_{\operatorname{fin}}\Bigg)
- \psi \bigg(\sum_{m=1}^{k}e_{\etabar\restl \langle m,n\rangle}\bigg)\\
& = & \sum_{m=1}^{k}\big( e_{\etabar\restl \langle m,n\rangle} - \psi ( e_{\etabar\restl \langle m,n\rangle})\big)
+ \Bigg( \bigg(c_n-\sum_{m=1}^{k}\varphi_\etabar\big(\etabar\restl \langle m,n\rangle\big)\bigg)e_n + C^\om_{\operatorname{fin}}\Bigg)\\
& = & \sum_{m=1}^{k}\big( \varphi\big(\etabar\restl \langle m,n\rangle\big)e_n+ C^\om_{\operatorname{fin}} \big)
+ \Bigg( \bigg(c_n-\sum_{m=1}^{k}\varphi_\etabar\big(\etabar\restl \langle m,n\rangle\big)\bigg)e_n + C^\om_{\operatorname{fin}}\Bigg)\\
& = & \bigg(c_n+\sum_{m=1}^{k}\Big( \varphi\big(\etabar\restl \langle m,n\rangle\big) - \varphi_\etabar\big(\etabar\restl \langle m,n\rangle\big)\Big) \bigg)e_n + C^\om_{\operatorname{fin}}\\
& = & c_n e_n + C^\om_{\operatorname{fin}}
\end{eqnarray*}
in $G_C$. From this we infer $x_n=(n+1)x_{n+1}+c_n$ in $C\subseteq \bar C \subseteq G_C$, contradicting~\eqref{3}. Hence, the aforementioned
exact sequence does not split, and $\Ext(F_C, C) \ne 0$ follows.
\end{proof}

\section{Final Remark} \label{5}
In Section \ref{4}, given any group $C$ which fails to be cotorsion, we chose cardinals
\[\l_1= \l_1^{\aleph_0} \ge |C| \quad \mbox{ and } \quad \l_{i+1} =2^{\l_i}\]
and used the $\bar\lambda$-Black Box for $\bar\lambda =\langle \l_1,\ldots,\l_k\rangle$
to construct an $\aleph_k$-free group $F_C$ with $\Ext(F_C, C) \ne 0$. It should be noted
that $B\subseteq F_C \subseteq \Bhat$ with $|B|=\l_k^{\aleph_0}=\l_k$ and
$|\Bhat|=|B|^{\aleph_0}=\l_k^{\aleph_0}=\l_k$. Thus we have $|F_C|=\l_k$,
and we actually can prove an even stronger statement as a natural extension of
Theorem \ref{mainGP} to $\aleph_k$-free groups.

\begin{lemma}[ZFC]
Let $\bar\lambda =\langle \l_1,\ldots,\l_k\rangle$ for $k\ge 2$ be a finite sequence of infinite cardinals with
\[\l_1= \l_1^{\aleph_0} \quad \mbox{ and } \quad \l_{i+1} =2^{\l_i}.\]
Then there exists an $\aleph_k$-free group $F$ of cardinality $|F|=\l_k$ such that for any group $C$
of cardinality $|C| \le \l_1$ the following statements are equivalent.
\begin{itemize}
\item[(i)] $C$ is cotorsion.
\item[(ii)] $\Ext(F,C) = 0$.
\end{itemize}
\end{lemma}

\begin{proof}
Again, (i) obviously implies (ii) as all $\aleph_k$-free groups are torsion-free.
Thus, we only need to verify the converse statement. To that effect, we must provide a suitable group $F$
such that $\Ext(F,C) \ne 0$ for every group $C$ of cardinality $|C|\le \l_1$ which fails to be cotorsion.

For this purpose define the family $\mathcal{D}$ of groups to contain one isomorphic copy of every
group $C$ of cardinality $|C|\le \l_1$ which fails to be cotorsion.
Note that \[ |\mathcal{D}|\le \l_1^{|\l_1 \times \l_1|}=2^{\l_1}=\l_2.\]
We now define
\[ F= \bigoplus_{D\in \mathcal{D}} F_D,\]
which is an $\aleph_k$-free group of cardinality $|F|=\l_2 \cdot \l_k=\l_k$. If now $C$ is any group of
cardinality $|C|\le \l_1$ which fails to be cotorsion, then we can find some $C \cong C' \in \mathcal{D}$,
and
\[ \Ext(F,C) = \Ext\big(\bigoplus_{D\in \mathcal{D}} F_D,C \big) = \prod_{D\in \mathcal{D}} \Ext(F_D,C) \ne 0\]
as $\Ext(F_{C'},C) \cong \Ext(F_{C'},C') \ne 0$.
\end{proof}


\begin{thebibliography}{9}

\bibitem{B}
R. Baer, \emph{Abelian groups without elements of finite order}, Duke Math. J. {\bf 3} (1937), 68--122.

\bibitem{Ek}
P. Eklof, \emph{Whitehead's problem is undecidable},
Amer. Math. Mon. {\bf 83}(10) (1976), 775--788.

\bibitem{EM}
P. Eklof, A. Mekler, \emph{Almost Free Modules, Set-theoretic
Methods}, North-Holland (2002).

\bibitem{GHSa}
R. G\"obel, D. Herden, H. G. Salazar Pedroza, \emph{$\aleph_k$-free separable groups with prescribed endomorphism ring},
Fund. Math. {\bf 231} (2015), 39--55.

\bibitem{GHS2}
R. G\"obel, D. Herden, S. Shelah, \emph{Prescribing endomorphism rings of  $\aleph_n$-free modules},
J. Eur. Math. Soc. {\bf 16} (2014), 1775--1816.

\bibitem{GS3}
R. G\"obel, S. Shelah,  \emph{$\aleph_n$-free modules with trivial dual},
Results in Math. {\bf 54} (2009), 53--64.

\bibitem{GSS}
R. G\"obel, S. Shelah, L. Str\"ungmann, \emph{$\aleph_n$-free modules over complete discrete valuation domains with almost trivial dual},
Glasg. Math. J. {\bf 55} (2013), 369--380.

\bibitem{GP}
R. G\"obel, R. Prelle, \emph{Solution of two problems on cotorsion abelian groups},
Arch. Math. {\bf 31} (1978), 423--431.

\bibitem{GT}
R. G\"obel, J. Trlifaj, \emph{Endomorphism Algebras and Approximations of Modules -- Vol. {\bf 1, 2}},
Expositions in Mathematics {\bf 41}, Walter de Gruyter Verlag, Berlin (2012).

\bibitem{Gr}
P. A. Griffith, \emph{$\aleph_n$-free abelian groups}, Quart. J.
Math. (Oxford) (2) {\bf 23} (1972), 417--425.

\bibitem{Habil}
D. Herden, \emph{Constructing $\aleph_k$-free Structures}, Habilitationsschrift, Universit\"at Duisburg-Essen (2013).

\bibitem{HSa}
D. Herden, H. G. Salazar Pedroza, \emph{Separable $\aleph_k$-free modules with almost trivial dual}, Comment. Math. Univ. Carolin. {\bf 57} (2016), 7--20.

\bibitem{Hi}
P. Hill, \emph{ New criteria for freeness in abelian groups II}, Trans.
Amer. Math. Soc. {\bf 196} (1974), 191--201.

\bibitem{Sh44}
{S. Shelah}, \emph{Infinite abelian groups, Whitehead problem and
some constructions}, Israel J. Math. \textbf{18} (1974), 243--256.

\bibitem{S3} S. Shelah, \emph{$\aleph_n$-free abelian groups with no
non-zero homomorphisms to $\Z$},  Cubo -- A Mathematical Journal {\bf 9} (2007), 59--79.

\bibitem{S4} S. Shelah, \emph{Quite free complicated abelian groups, PCF and Black Boxes},  submitted, arXiv:1404.2775.

\bibitem{Sp}
E. Specker, \emph{Additive Gruppen von Folgen ganzer Zahlen}, Port. Math. {\bf 9} (1949), 131--140.

\bibitem{St}
K. Stein, \emph{Analytische Funktionen mehrerer komplexer Ver\"anderlichen zu
vorgegebenen Periodizit\"atsmoduln und das zweite Cousinsche Problem}, Math. Ann. {\bf 123} (1951), 201--222.

\end{thebibliography}
\end{document}